\documentclass[12pt]{amsart}
\usepackage{hyperref}
\usepackage{amsmath}
\usepackage{amsthm}
\usepackage{amsfonts}
\usepackage{amssymb}

\usepackage{mathptmx} % rm & math
\usepackage[scaled=0.90]{helvet} % ss
\usepackage{courier} % tt
\normalfont
\usepackage[T1]{fontenc}
\usepackage{color}
\usepackage{graphicx}
\DeclareGraphicsExtensions{.eps}

\newtheorem{theorem}{Theorem}[section]
\newtheorem{utv*}{Proposition}
\newtheorem{hyp*}{Conjecture}
\newtheorem{lemma}[theorem]{Lemma}

\newtheorem{defin}{Definition}
\newtheorem{zamech}{Remark}
\newtheorem*{th*}{Theorem}

\newcommand{\av}[2]{\langle #1\rangle_{_{\scriptstyle #2}}}
\newcommand{\ave}[1]{\langle #1\rangle}

\def\sli{\sum\limits}
\def\ili{\int\limits}

\def\la{\lambda}

\def\R{\mathbb{R}}

\def\ep{\varepsilon}
\def\vf{\varphi}

\def\Om{\Omega}

\def\cD{\mathcal{D}}

\newcommand{\cB}{\mathcal{B}}
\newcommand{\cG}{\mathcal{G}}

\newcommand{\al}{\alpha}

\newcommand{\cz}{Calder\'{o}n--Zygmund\ }

\newcommand{\pd}{\partial}

\def\cyr{\fontencoding{OT2}\fontfamily{wncyr}\selectfont}
\DeclareTextFontCommand{\textcyr}{\cyr}
{\end{list}}

\newcounter{vremennyj}

\begin{document}
\title[Bellman approach to the one-sided bumping]{Bellman approach to the one-sided bumping for weighted estimates of Calder\'on--Zygmund operators}
%\author{Alexander Volberg}
\author{Fedor Nazarov}
\address{Department of Mathematics, University of Wisconsin-Madison and Kent Sate University}
\email{nazarov@math.wisc.edu}
\thanks{Work of F.~Nazarov is supported  by the NSF grant}

\author{Alexander Reznikov}
\address{Department of Mathematics,  Michigan State University, East
Lansing, MI 48824, USA}
\email{reznikov@ymail.com}
\author{Alexander Volberg}
\thanks{Work of A.~Volberg is supported  by the NSF under the grant  DMS-0758552.
}
\address{Department of Mathematics, Michigan State University, East
Lansing, MI 48824, USA}
\email{volberg@math.msu.edu}
\urladdr{http://sashavolberg.wordpress.com}

\makeatletter
\@namedef{subjclassname@2010}{
  \textup{2010} Mathematics Subject Classification}
\makeatother


% 42B	Harmonic analysis in several variables
% 42B20	Singular and oscillatory integrals (Calder?on-Zygmund, etc.)
% 42B35	Function spaces arising in harmonic analysis

% 47A	General theory of linear operators
% 47A30	Norms (inequalities, more than one norm, etc.)

%{30E20, 47B37, 47B40, 30D55.}
%
% 30D55	$H^p$-classes (1980-2009)
% 30E20	Integration, integrals of Cauchy type, integral representations of analytic functions
%
% 47B   	Special classes of linear operators
% 47B37	Operators on special spaces (weighted shifts, operators on sequence spaces, etc.)
% 47B40	Spectral operators, decomposable operators, well-bounded operators, etc.


   \begin{abstract}
We give again a proof of  weighted estimate of any \cz operator.  This is  under a universal sharp sufficient condition that is weaker than the so-called bump condition.  Bump conjecture was recently solved independently and simultaneously by A. Lerner and Nazarov--Reznikov--Treil--Volberg. The latter paper uses the Bellman approach. Immediately a very natural and seemingly simple question arises how to to strengthen the bump conjecture result by weakening its assumptions in a natuarl symmetric way. This is what we are dealing with here. However we meet an unexpected and, in our opinion, deep obstacle, that allows us to make only partial result.
Our proof consists of two main parts: reduction to a simple model operator, construction of Bellman function for estimating this simple operator. The newer feature is that the domain of definition of our Bellman function is infinitely dimensional.
\end{abstract}

\date{}
\maketitle

\section{Introduction}
In this paper we consider a question about the sufficiency of certain ``bump'' conditions for the boundedness of all Calder\'on-Zygmund operators. Precisely, we consider two functions $u, v$, positive almost everywhere, and ask a question:
\begin{equation}
\label{mainquestion}
\begin{aligned}
&\text{ When there  exists a constant}\; C,\;\text{ such that for every function} 
\\
&\qquad\qquad\qquad f\in C_{0}^{\infty} \quad \|T(fu)\|_{L^2(v)}\leqslant C \|f\|_{L^2(u)}?
\end{aligned}
\end{equation}
Of course the constant $C$ is assumed to be independent of $f$.
The famous ``joint $A_2$'' condition, necessary but not sufficient, was introduced by D. Sarason. It looked like this:
$$
\mbox{There exists a constant $C$, such that for any interval $I$ the following holds:}
$$
$$
\frac{1}{|I|}\ili_I u\,dx\; \;\frac{1}{|I|}\ili_I v\,dx \leqslant C\,.
$$

We would like to rewrite this condition in the following way:
\begin{equation}
\label{a2}
\|u\|_{L^1(I, \frac{dx}{|I|})} \cdot \|v\|_{L^1(I, \frac{dx}{|I|})} \leqslant C\,.
\end{equation}

It is well known that this condition is not sufficient for the boundedness of $T$ for interesting $T$ (like the Hilbert transform, or a dyadic shift). So, we want to consider a bigger left-hand side, to make the condition stronger. Thus, instead of the $L^1$-norm, we would like to put something bigger. This brings us to the notion of Orlicz norms.
\subsection{Orlicz norms}
Consider a function $\Phi$ that is increasing, and convex. Then define
$$
\|u\|_{L^\Phi_I} = \inf\{\la \colon \frac{1}{|I|}\ili_I \Phi\left(\frac{f(t)}{\la}\right) dt \leqslant 1\}. 
$$
Notice that $\Phi(t)=t$ gives the normalized $L^1$ norm, and $\Phi(t)=t^p$ gives the normalized $L^p$ norm.

\subsection{History of the question}
\label{history}

An interesting ``bump" conjecture was open for quite a while, and it was recently 
solved independently and at the same time by two quite distinct (but having some 
fundamental similarity) methods . One solution, due to Andrei Lerner, uses {\it local sharp maximal function} approach, see \cite{Le}. Another solution, due to 
Nazarov--Reznikov--Treil--Volberg \cite{NRTV} used a Bellman function technique, 
but with a new twist, the Bellman function of \cite{NRTV} depends on infinitely many variables. This bump conjecture is the statement that replaces
the left hand side of \eqref{a2} by  its ``bumped-up" version, conjecturing that this version is now sufficient for the boundedness of all \cz operators. Thus, it introduces the bumped $A_2$ condition:
\begin{equation}
\label{ba2}
\|u\|_{L^\Phi_I} \cdot \|v\|_{L^\Phi_I} \leqslant C\,.
\end{equation}
Of course $\Phi(t)=t$ is just the same as \eqref{a2}, so we need some condition (preferably sharp) on $\Phi$ 
to ensure the boundedness of interesting (actually of all) \cz operators. 
This condition (and this is known to be sharp) was invented by Carlos P\'erez and David Cruz-Uribe, \cite{CUP99}, \cite{CUPpisa}, and it is
\begin{equation}
\label{intPhi}
\Phi \;\text{is convex increasing function such that}\;\int_1^\infty \frac{dt}{\Phi(t)} <\infty\,.
\end{equation}

The bump conjecture itself (see \cite{CUP99}, \cite{CUPpisa}, \cite{CUP00b}, \cite{CUMaP07}, \cite{CUMaP12}, \cite{CUMaPbook}) reads now: given that two weights $u, v$ satisfy \eqref{ba2} and $\Phi$ in \eqref{ba2} satisfies \eqref{intPhi}, prove that
any \cz operator is bounded from $L^2(u)$ into $L^(v)$ in the sense \eqref{mainquestion} stated above.

This has been proved, as  we already mentioned in \cite{Le} and in \cite{NRTV}. 
However, a very natural conjecture is that \eqref{ba2} can be weakening even more. 
Namely, we want to bump only one weight at a time. We get the following quite natural one-sided bump assumption:
\begin{equation}
\label{1sba2}
\begin{aligned}
&\mbox{There exists a constant $C$, such that for any interval $I$ the following holds:} 
\\
&\|u\|_{L^\Phi_I} \cdot \|v\|_{L^1(I, \frac{dx}{|I|})} \leqslant C, \\
&\mbox{and}\\
&\|u\|_{L^1(I, \frac{dx}{|I|})} \cdot \|v\|_{L^\Phi_I} \leqslant C.
\end{aligned}
\end{equation}

And now one-sided bump conjecture is the following statement:
suppose \eqref{1sba2} holds for all intervals (cubes), and suppose $\Phi$ 
satisfies integrability condition \eqref{intPhi}, then any \cz operator is bounded from $L^2(u)$ into $L^2(v)$ in the sense \eqref{mainquestion} stated above.

The attempt to prove this has been done in \cite{CURV}. But we could do this only for {\it some} $\Phi$. 
The present article is another attempt. It is sort of different in technique, it uses a Bellman function method unlike \cite{CURV} that used a stopping time argument.

The present approach is slightly more propitious, because Bellman technique is ``reversible". 
We actually manage to prove by Bellman technique only slightly more general result than in \cite{CURV}. 
However, this ``reversibility" feature probably indicates that the one sided bump conjecture in full generality (for all $\Phi$ subject to \eqref{intPhi}) might be actually wrong.

Recently M. Lacey \cite{L} using a parallel corona argument generalized the results of this paper to the case $p\not=2$ and to a more general bump condition. 
\section{A construction from ~\cite{NRTV}}
\label{co}
To formulate the main result we use a certain language.

For that we need the following construction. Define a function $\Psi$ in the following parametric way:
$$
\begin{cases} 
s=\frac{1}{\Phi(t)\Phi'(t)} \\
\Psi(s):=\Phi'(t).
\end{cases}
$$
Of course, we define $\Psi$ in this way near $s=0$. 

We give the following definition.
\begin{defin}\label{regbumps}
A function $\Phi$ is called \textit{regular bump}, if for any function $u$ there holds
$$
\|u\|_{L^\Phi_I}\geqslant C\int N_I(t) \Psi(N_I(t))dt.
$$
\end{defin}
\begin{zamech}[\cite{NRTV}]
An example of regular bump is the following: $\Phi(t)=t\rho(t)$, and
$$
t\frac{\rho'(t)}{\rho(t)}\log\rho(t) \to 0, \; \mbox{as} \; t\to \infty.
$$
\end{zamech}
The important result is the following.
\begin{lemma}
The function $s\mapsto\Psi(s)$ is decreasing; the function $s\mapsto s\Psi(s)$ is increasing; the function $\frac{1}{s\Psi(s)}$ is integrable near $0$. 
Moreover, the following inequality is true with a uniform constant $C$ (which may depend only on $\Phi$):
$$
C\|u\|_{L^\Phi_I}\geqslant \int N_I(t) \Psi(N_I(t))dt,
$$
where 
$$
N_I(t)=\frac{1}{|I|} |\{x\in I\colon u(x)\geqslant t\}|.
$$

Further, for ``regular'' functions $\Phi$ we have that
$$
\|u\|_{L^\Phi_I}\sim \int N_I(t) \Psi(N_I(t))dt.
$$
\end{lemma}

\section{The main results. Boundedness and weak boundedness.}
\label{main}
Given a function $\Phi$, satisfying \eqref{intPhi}, build the corresponding function $\Psi$ as in Section \ref{co}.  We prove the following theorems. Regularity conditions are not very important, but the last condition in the statement of the theorem is actually an important restriction. This is the restriction one would wish to get rid of. Or to prove that it is actually needed. Lately we believe  that one cannot get rid of it.
We give a non-standard definition.

\begin{defin}\label{weaklyconc}
A function $f$ is ``weakly concave'' on its domain, if for any numbers $x_1, \ldots, x_n$ and $\la_1, \ldots, \la_n$, such that $0\leqslant \la_j\leqslant 1$, 
and $\sli \la_j = 1$, the following inequality holds:
$$
f(\sli \la_j x_j)\geqslant C \sli \la_j f(x_j),
$$
where the constant $C$ does not depend on $n$.
\end{defin}

\begin{theorem}
\label{propep}
Suppose there exists a function $\Phi_0$ with corresponding $\Psi_0$, such that:
\begin{itemize}
\item $\Phi_0$  satisfies \eqref{intPhi};
\item $\Phi$ and $\Phi_0$ are regular bumps;
\item There is a function $\ep$, such that $\Psi_0(s)\leqslant C \Psi(s) \ep(\Psi(s))$;
\item The function $t\mapsto t\ep(t)$ is weakly concave, in the sense of the Definition \ref{weaklyconc};
\item The function $t\mapsto  t\ep(t)$ is strictly increasing near $\infty$;
\item The function $t\mapsto  t\ep(t)$ is concave near $\infty$;
\item The function $t\mapsto \frac{\ep(t)}{t}$ is integrable at $\infty$.
\end{itemize}
Suppose that there exists a constant $C$, such that a one-sided bump condition \eqref{1sba2} holds.
Then any \cz operator is bounded  from $L^2(u)$ into $L^{2}(v)$ in the sense of \eqref{mainquestion}. 
\end{theorem}

\begin{theorem}
Suppose the function $\Phi$ satisfies all conditions from the theorem above. Suppose that there exists a constant $C$, such that
$$
\|u\|_{L^1(I, \frac{dx}{|I|})} \cdot \|v\|_{L^\Phi_I} \leqslant C.
$$
Then any Calderon-Zygmund operator is weakly bounded  from $L^2(u)$ into $L^{2, \infty}(v)$, i.e. there exists a constant $C$, such that for any function $f\in C_0^\infty$ there holds
\begin{equation}\label{weakbdd}
\|T(fu)\|_{L^{2, \infty}(v)} \leqslant C \|f\|_{L^2(u)}.
\end{equation}
\end{theorem}

\section{Examples of $\Phi$ satisfying the restrictions of the main results: the cases from \cite{CURV}}
\label{examples}

The biggest difference of the above results with those of \cite{CURV} is that here we gave the integral condition on the corresponding bump function $\Phi$. 
To compare with \cite {CURV} we notice that in \cite{CURV} theorems above were proved in two cases:
\begin{enumerate}
\item $\Phi(t)=t\log^{1+\sigma}(t)$; 
\item $\Phi(t)=t\log(t)\log\log^{1+\sigma}(t)$, for sufficiently big $\sigma$.
\end{enumerate}
We show that these results are covered by our theorems.

First, suppose $\Phi(t)=t\log^{1+\sigma}(t)$. Then $\Psi(s)\asymp\log^{1+\sigma}(\frac1s)$. We put $\Phi_0(s)=t\log^{1+\frac\sigma 2}(t)$, and then $\ep(t)=t^{-\frac{\sigma}{2(1+\sigma)}}$.
Then, clearly, all properties of $\ep$ from our theorem are satisfied.

\bigskip

Next, suppose $\Phi(t)=t\log(t)\log\log^{1+\sigma}(t)$. Then $\Psi(s)\asymp\log(\frac1s)\log\log^{1+\sigma}(\frac1s)$. We put $\Phi_0(t)=t\log(t)\log\log^{1+\delta\sigma}(t)$, $\delta<1$ which gives $\ep(t)=\log^{-(1-\delta)\sigma}(t)$. 
Then, the integral $\int^\infty \frac{\ep(t)}{t}\,dt$ converges if  $\sigma>1$, and we choose $\delta$ to be very small. 
%The weak concavity is proved (for $\delta=\frac12$) in the Section \ref{exloglog}. Other properties of $\ep$ are an easy exercise on differentiation.

\bigskip

Moreover, examining the proof of Theorem 5.1 from \cite{CURV}, we get the result from our paper but with a condition 
$$
\mbox{The function $t\mapsto \frac{\sqrt{\ep(t)}}{t}$ is integrable at $\infty$.}
$$
We notice that for regular functions we have $\ep(t)\to 0$ when $t\to \infty$, and so $\ep(t)<\sqrt{\ep(t)}$. Thus, our results work for more function $\ep$ and, thus, bumps $\Phi$.
\section{Preliminary results}
\label{prelim}
In this section we state two helpful results. They are important building blocks in our proof. The first result is due to A. Lerner. It treats the so called {\it Banach function spaces}, see the definition in \cite{Le}. 
\begin{theorem}[Theorem 1.1 from \cite{Ler}]\label{Ler}
Suppose $X$ is a {\it Banach function space} over $\R$ equipped with Lebesgue measure. Then, for any appropriate $f$, 
$$
\|Tf\|_X \leqslant C(T, n) \sup_{\mathcal{D}, \{a_I\}}\|T_{\mathcal{D}, \{a_I\}}|f|\|_X,
$$
where the supremum is taken over all dyadic lattices $\mathcal{D}$ and all sequences $\{a_I\}$, such that $\sli_{I\subset J, I\in \mathcal{D}}a_I|I|\leqslant 2|J|$ for any $J\in \mathcal{D}$. Here
$$
T_{\mathcal{D}, \{a_I\}}f = \sli_I a_I \av{f}{I}\chi_I.
$$
\end{theorem}
Here 
$$
\av{f}{I} := \frac{1}{|I|} \ili_I f dx\,.
$$
%%%%%%%%
%\begin{zamech}
%It is not hard to see that Lerner's theorem holds for $X=L^{2, \infty}$, even though it is not a Banach space.
%\end{zamech}
%%%%%%%%%

The next result is the famous testing conditions type theorem. We state it in the way we will use it. First let us introduce the small notation
$$
u(J) := \int_J u\, dx\,.
$$
Notice that 
$$
u(J)= \|1_J\|_{L^2(u)}^2\,.
$$

\begin{theorem}\label{Laceyandco}
Suppose there exists a constant $C$, such that for any dyadic interval $J\in \mathcal{D}$ there holds
\begin{equation}
\label{test}
\begin{aligned}
&\|\chi_J T_{\mathcal{D}, \{a_I\}}(u\chi_J)\|_{L^2(v)}^2 \leqslant C\, u(J),
\\
&\|\chi_J T_{\mathcal{D}, \{a_I\}}(v\chi_J)\|_{L^2(u)}^2 \leqslant C\, v(J)\,.
\end{aligned}
\end{equation}
Then the operator $T_{\mathcal{D}, \{a_I\}}$ is bounded in the sense of \eqref{mainquestion}.

Moreover, if the weights $u$ and $v$ satisfy the joint $A_2$ condition, meaning that for any interval $J$ there holds $\av{u}{J}\av{v}{J}\leqslant C$, and there holds 
$$
\|\chi_J T_{\mathcal{D}, \{a_I\}}(v\chi_J)\|_{L^2(u)}^2 \leqslant C\, v(J)\,,
$$
then the operator $T_{\mathcal{D}, \{a_I\}}$ is weakly bounded in the sense of \eqref{weakbdd}.
\end{theorem}

The first part can be found in \cite{NTV-mrl}. The last statement follows from the Theorem 4.3 of \cite{HLMORSUT} and Corollary 3.2 of \cite{HyLa}. It also needs the known fact that the maximal function is weakly bounded if weights satisfy the joint $A_2$ condition.

\subsection{Self improvements of Orlicz norms.}
In this section we prove a technical result, which has the following ``hand-waving'' explanation: suppose we take a function $\Phi$ and a smaller function $\Phi_0$. We explain how small can be the quotient $\frac{\|u\|_{L^{\Phi_0}_I}}{\|u\|_{L^{\Phi}_I}}$ in terms of smallness of $\frac{\Phi_0}{\Phi}$. In what follows we consider only ``regular bumps'' functions in the sense of the Definition \ref{regbumps}.

Suppose we have two functions $\Phi$ and $\Phi_0$, and we have built functions $\Psi$ and $\Psi_0$. We suppose that 
$$
\Psi_0(s)\leqslant C \Psi(s) \ep(\Psi(s)).
$$

The following theorem holds.
\begin{theorem}
\label{selfimp}
Let $I$ be an arbitrary interval (cube). If a function $t\mapsto t\ep(t)$ is weakly concave, then 
$$
\|u\|_{L_I^{\Phi_0}}\leqslant C \|u\|_{L_I^{\Phi}} \;\ep\Big(\frac{\|u\|_{L_I^\Phi}}{\av{u}{I}}\Big)\,.
$$
\end{theorem}

To do that we need the following easy lemma:
\begin{lemma}
For weakly concave functions the Jensen inequality holds with a constant:
$$
\ili f(g(t))d\mu(t) \leqslant C f(\ili g(t)d\mu(t)).
$$
\end{lemma}
\begin{proof}
This is true since if $g$ is a step function, then this is just a definition. Then we pass to the limit. 
Here we essentially used that we can take a convex combination of $n$ points, and the constant in the definition above does not depend on $n$.
\end{proof}

\begin{proof}[Proof of the Theorem]
In the proof we omit the index $I$.
Since for regular bumps we know that 
$$
\|u\|_{L^\Phi} \sim \int \Psi(N(t))N(t)dt, 
$$
we simply need to prove that
$$
\int \Psi_0(N(t))N(t)dt \leqslant C \int \Psi(N(t))N(t)dt\; \ep\Big(\frac{\int \Psi(N(t))N(t)dt}{\int N(t)dt}\Big)
$$

Our first step is the obvious estimate of the left-hand side:
$$
\int \Psi_0(N(t))N(t)dt\leqslant C \int \Psi(N(t))\ep(\Psi(N(t))N(t)dt\,.
$$

Denote $a(t)=t\ep(t)$. Then we need to prove that
$$
\int a(\Psi(N(t)) N(t)dt \leqslant C \int N(t)dt \; a\Big(\frac{\int \Psi(N(t))N(t)dt}{\int N(t)dt}\Big)\,.
$$
We denote 
$$
d\mu = \frac{N(t)}{\int N(t)dt}dt,
$$
it is a probability measure. Moreover, by assumption, $t\mapsto a(t)$ is concave. Therefore, by Jensen's inequality (from the Lemma),
$$
\int a(f(t))d\mu(t) \leqslant C a\Big(\int f(t)d\mu(t)\Big)\,.
$$
Take $f(t)=\Psi(N(t))$, and the result follows.
\end{proof}
\subsection{Examples}
\subsubsection{$\log$-bumps}
First, if $\Phi(t)=t\log^{1+\sigma}(t)$, then $\Psi(s)=\log^{1+\sigma}(1/s)$, and 
$$
\frac{\Psi_0(s)}{\Psi(s)}=\log^{-\frac{\sigma}{2}}(1/s) = \Psi^{-\frac{\sigma}{2(1+\sigma)}}.
$$
Thus, $\ep(t)=t^{-\frac{\sigma}{2(1+\sigma)}}$, and everything is fine.

\bigskip
\subsubsection{$\log\log$-bumps}\label{exloglog}
Next example is with double logs. In fact, when
$$
\Psi(s)=log(1/s)(\log\log(1/s))^{1+\sigma}, \;\;\;\; 
\Psi_0(s)=log(1/s)(\log\log(1/s))^{1+\sigma/2}
$$
then
$$
\frac{\Psi_0(s)}{\Psi(s)}=\log\log^{-\sigma/2}(1/s) \sim (\log(\Psi(s)))^{-\sigma/2}.
$$
Thus, $\ep(t)=(\log t)^{-\frac{\sigma}2}$. Everything would be also fine, except for one little thing: the function $t\mapsto t\ep(t)$ is concave on infinity, but not near $1$.  However, $t\mapsto t\ep(t)$ is weakly concave on $[2, \infty)$, and this is enough for our goals as without loss of generality, $\Psi(s)\geqslant 2$. 

 So let us prove that $a(t)= t\ep(t)$ is weakly concave on $[2, \infty)$.

Let $\varkappa:= \frac{\sigma}2$. The function $a$ has a local minimum at $e^{\varkappa}$ and its concavity changes at $e^{\varkappa+1}$. We now take $x_j$, $\la_j$ and $x=\sli \la_j x_j$. We first notice that if $x>e^{\varkappa+1}$, the we are done, because then $(x,\sli \la_j a(x_j))$ lies under the graph of $a$. 

If $2\le x<e^{\varkappa+1}$, then $a(x)>\min_{[2, e^{\varkappa+1}]}a = c(\varkappa)$. Moreover, if $\ell$ is a line tangent to graph of $a$, starting at $(2, a(2))$, and $\ell$ ``kisses'' the graph at a point $(r, a(r))$, then $\sli \la_j a(x_j)\leqslant a(r)=c_1(\varkappa)$. This follows from the picture: a convex combination of $a(x_j)$ can not be higher than this line.

Therefore, 
$$
a(\sli \la_j x_j)\geqslant c(\varkappa)\geqslant C c_1(\varkappa) \geqslant \sli \la_j a(x_j)\,.
$$
This finishes our proof.

\subsection{Proof of the main result: notation and the first reduction.}
We fix a dyadic grid $\mathcal{D}$. Theorems \ref{Ler} and \ref{Laceyandco} show that to prove our main results it is enough to show that the following implication holds:
$$
\mbox{if for all}\, \,\,I \, \; \|u\|_{L^\Phi_I} \cdot \|v\|_{L^1(I, \frac{dx}{|I|})} \leqslant B_{u,v} \; \mbox{then} \; \|\chi_J T_{\mathcal{D}, \{a_I\}}(u\chi_J)\|_{L^2(v)}^2\leqslant C\, u(J) ,
$$
where $C$ does not depend neither on the grid, nor on the sequense $\{a_I\}$. It can, of course, depend on $B_{u,v}$. This will prove the weak bound $T: L^2(v)\rightarrow L^{2,\infty}(u)$.
For simplicity, we denote $T_a = T_{\mathcal{D}, \{a_I\}}$. It is an easy calculation that, under the joint $A_2$ condition (which is definitely satisfied under the bump condition), it is enough to get an estimate of the following form:
\begin{equation}
\label{glav}
\frac{1}{|J|}\sli_{J\subset I} a_I \cdot \av{u}{I} \cdot \frac{1}{|I|}\sli_{K\subset I}a_K \av{u}{K}\av{v}{K}|K| \cdot |I|\leqslant C\, u(J)\,.
\end{equation}

\begin{zamech}
By the rescaling argument it is clear that we can assume $B_{u,v}$ as small as we need (where ``smallness'', of course, depends only on the function $\Phi$). We need this remark, since all behaviors of our function $\ep$ are studied near $0$. 
\end{zamech}

\begin{zamech}
Everything is reduced to \eqref{glav}. We concentrate on proving \eqref{glav}. Clearly, by scale invariance, it looks very tempting to make \eqref{glav} a Bellman function statement. This will be exactly our plan from now on.
\end{zamech}

\section{Bellman proof of \eqref{glav}: introducing the ``main inequality''}
\label{mainineq}
%We need the following lemma.
%
%We denote the constant from the Lemma by $P$, not by $C$, since we will need it in the next theorem. 
We start this Section with the following notation. We fix two weights $u$ and $v$, and a Carleson sequense $\{a_I\}$. We denote
$$
u_I = \av{u}{I}, \; \; v_I=\av{v}{I}; \; \; \; N_I(t)=\frac1{|I|} |\{x\colon u(x)\geqslant t\}|;
$$
$$
A_I = \frac{1}{|I|}\sli_{J\subset I}a_J |J|; 
$$
$$
L_I = \frac{1}{|I|}\sli_{J\subset I}a_J \av{u}{J}\av{v}{J} |J|.
$$

We proceed with two theorems that prove our main result. Everywhere in the future we use that   $\av{u}{I}\av{v}{I}=u_I v_I \leqslant \delta < 1$ for any $I$. We can do it due to simple rescaling.

\begin{theorem}
\label{mainineq}
Suppose that 
$$
\frac{\Psi_0}{\Psi}\leqslant \ep (\Psi),
$$
where $\ep$ satisfies properties of Theorem \ref{main}, from which the main one is
\begin{equation}
\label{intep}
\int^\infty \frac{\ep(t)}{t}\, dt<\infty\,.
\end{equation}
Let $\delta$ be small enough, and
$$
\Om_1 = \{(N, A) \colon 0\leqslant N \leqslant 1; \; 0\leqslant A \leqslant 1\} %; \; v\geqslant 0; \; L\geqslant 0\},
$$
and for some constant $P$
$$
\Om_2 = \{(u,v,L,A)\colon 0\leqslant A \leqslant 1; \; u,v,L \geqslant 0; \; uv\leqslant \delta; \; L\leqslant P\cdot \sqrt{uv}\}.
$$
Suppose we have found a function $B_1$, defined on $\Om_1$, and a function $B_2$, defined on $\Om_2$, such that:
\begin{align}
& 0\leqslant B_1\leqslant N;\\
& (B_1)'_A \geqslant 10\frac{N}{\Psi_0(N)};\\
%& (B_1)'_A \geqslant 0; \\
& -d^2 \,B_1 \geqslant 0;\\
& 0\leqslant B_2\leqslant u;\\
& (B_2)'_A \geqslant 0\\ 
&(B_2)'_A\geqslant c\cdot u\cdot L, \; \mbox{when} \; P\cdot \sqrt{uv}\ge L\geqslant \frac{uv}{\ep(\frac{1}{uv})}; \\
&uv(B_2)'_L \geqslant -\delta_1 uL, \; \mbox{for sufficiently small $\delta_1$ in the whole of $\Omega_2$}; \\
& -d^2\,B_2 \geqslant 0.
\end{align}

Then for the function of an interval $\mathcal{B}(I):= B_2(u_I, v_I, L_I, A_I)+\ili_0^\infty B_1(N_I(t),  A_I)dt$ the following holds:
\begin{align}
& 0\leqslant \mathcal{B}(I) \leqslant 2u_I\\
&\mathcal{B}(I)-\frac{\mathcal{B}(I_+)+\mathcal{B}(I_-)}{2} \geqslant C\, a_I\cdot u_I\cdot L_I. \label{diff}
\end{align}
\end{theorem}

Next, we state
\begin{theorem}
\label{thenglav}
If such two functions $B_1$ and $B_2$ exist, then \eqref{glav} holds, namely
$$
\frac{1}{|I|}\sli_{J\subset I} a_I \cdot \av{u}{J} \cdot \frac{1}{|J|}\sli_{K\subset J}a_K \av{u}{K}\av{v}{K}|K| \cdot |I|\leqslant R^2 \ili_I u.
$$
\end{theorem}

\begin{proof}[Proof of the Theorem \ref{thenglav}]
This is a standard Green's formula applied to function $\mathcal{B}(I)$ on the tree of dyadic intervals. Let us explain the details.

Since the function $\mathcal{B}$ is non-negative, we have that
$$
2|I| u_I \geqslant |I|\mathcal{B}(I) \geqslant |I|\mathcal{B}(I) - \sli_{k=1}^{2^n} |I_{n,k}| \mathcal{B}(I_{n,k}).
$$
Here $n$ is fixed, and $I_{n,k}$ are $n-th$ generation descendants of $I$. Clearly, all $|I_{n,k}|$ are equal to $2^{-n}$. 

Let us denote $\Delta(J) = |J|\mathcal{B}(J) - |J_+|\mathcal{B}(J_+) - |J_-|\mathcal{B}(J_-)$, where $J_\pm$ are children of $J$. By the property \eqref{diff} we know that $\Delta(J)\geqslant C|J|\, a_J u_J L_J$.
By the telescopic cancellation, we get that
$$
|I|\mathcal{B}(I) - \sli_{k=1}^{2^n} |I_{n,k}| \mathcal{B}(I_{n,k}) = %(|I|\mathcal{B}(I) - |I_{1,1}|\mathcal{B}(I_{1,1})-|I_{1,2}|\mathcal{B}(I_{1,2})) + ( \sli_{k=1}^{2^n} |I_{k,n}| \mathcal{B}(I_{k,n})
\sli_{m=0}^{n-1} \sli_{k=1}^{2^m}\Delta(I_{m,k}).
$$
Combining our estimates, we get
$$
2|I| u_I \geqslant C \sli_{m=0}^{n-1} \sli_{k=1}^{2^m} |I_{m,k}| a_{I_{m,k}}u_{I_{m,k}}L_{I_{m,k}} = C \sli_{J\subset I, |J|\geqslant 2^{-n}|I|}|J|a_J u_J L_J.
$$
This is true for every $n$, with the constant $C$ independent of $n$. Thus, 
$$
u_I\geqslant C \frac{1}{|I|}\sli_{J\subset I} a_J u_J L_J |J|.
$$ 
The result follows from the definition of $L_J$.
\end{proof}

In the future we use the following variant of Sylvester criterion of positivity of matrix.

\begin{lemma}
\label{Sy}
Let $M=(m_{ij})_{i, j=1}^3$ be a $3\times 3$ real symmetric matrix such that $m_{11}<0$, $m_{11} m_{22}-m_{12}m_{21} >0$, and $\det \,M=0$. Then $M$ is nonpositive definite.
\end{lemma}

\begin{proof}
Let $E$ be a matrix with all entries being $0$ except for $e_{33}=1$. Consider $t>0$ and $A:=A(t):=M+tE$. It is easy to see that $a_{11}<0$, $a_{11} a_{22}-a_{12}a_{21} >0$, and $\det\, A=t\cdot  (m_{11} m_{22}-m_{12}m_{21} )>0$ when $t>0$. By Sylvester criterion, matrices $A(t)$, $t>0$, are all negatively definite. Therefore, tending $t$ to $0+$, we obtain, that $M$ is nonpositive definite.
\end{proof}

We need the following lemma, which is in spirit of \cite{VaVo}.
\begin{lemma}
Let $L_I$ be given by
$$
L_I = \frac{1}{|I|}\sli_{J\subset I}a_J \av{u}{J}\av{v}{J} |J|.
$$
Let $A_I$ given by 
$A_I = \frac{1}{|I|}\sli_{J\subset I}a_J |J|$. Suppose that it is bounded by $1$ for any dyadic $I$ (Carleson condition).
If for any dyadic interval $I$ we have that $\av{u}{I} \av{v}{I} \leqslant 1$, then it holds that for any  dyadic interval $I$ we have $L_I \leqslant P \sqrt{\av{u}{I} \av{v}{I}}$.
\end{lemma}
\begin{proof}
It is true since the function $T(u, v, A)=100\sqrt{uv}-\frac{uv}{A+1}$ is concave enough in the domain $G:=\{0\leqslant A\leqslant 1, \; uv < 1, \; u,v\geqslant 0\}$.  One can adapt the proof from \cite{VaVo}. 

First, we need to check that the function $T(x,y,A)$ is concave in $G$. Clearly, $T''_{A,A} < 0$. Next, 
\begin{equation}
\label{detT}
\det \begin{pmatrix} T''_{A,A} & T''_{A,v} \\ T''_{A, v} & T''_{v,v}
\end{pmatrix} = \frac{x}{y(A+1)^4} \cdot (50(A+1)\sqrt{xy}-xy) > 0\,.
\end{equation}
This expression is non-negative, because $A+1\geqslant 1$, and $\sqrt{uv}\leqslant 1$.
Finally,
$$
\det \begin{pmatrix} T''_{A,A} & T''_{A,v} & T''_{A,u} \\ T''_{A, v} & T''_{v,v} &  T''_{v,u} \\  T''_{A,u} &  T''_{v,u} &  T''_{u,u}
\end{pmatrix} = 0.
$$

Therefore, by Lemma \ref{Sy} we conclude that $T(u,v, A)$ is a concave function.

Next, 
$$
T'_A = \frac{uv}{(A+1)^2} \geqslant \frac{1}{4} uv.
$$

Thus, if we fix three points $(u, v, A)$, $(u_\pm, v_\pm, A_\pm)$, such that $u=\frac{u_+ + u_-}{2}$, $v=\frac{v_+ + v_-}{2}$, and $A=\frac{A_+ + A_-}{2} + a$, we get by the Taylor formula:
$$
T(u,v,A) - \frac{T(u_+, v_+, A_+)+T(u_-, v_-, A_-)}{2} \geqslant a T'_A(u,v,A) \geqslant C\, a\cdot uv.
$$
This requires the explanation. The Taylor formula we used  has a remainder with the second derivative at the intermediate point $P_\pm$ on segments $S_+ :=[(u,v, \frac{A_-+A_+}{2}),  (u,v, A_+)],$ $S_-:= [(u,v, \frac{A_-+A_+}{2}),  (u,v, A_-)]$.  One of this segments definitely lies inside domain $G$, where $T$ is concave, and this remainder will have the right sign. However the second segment can easily stick out of domain $G$, because $G$ itself is not convex. But notice that if, for example, $S_+$ is not inside $G$, still $(x,y, B)\in S_+$ implies that one of the coordinates, say $x$, must be smaller than $u$. Then $y$ can be bigger than $v$, but not much. In fact, 
$$
v_+- v= v-v_-\,\Rightarrow \, v_+ \le 2v-v_-\le 2v\,.
$$
Therefore, $y\le v_+\le 2v$. Then we have that $xy\le 2uv \le 2$. Let us consider $\widetilde{G}:=\{(x,y, A): 0\le A\le 1,\; x, y\ge 0,\; 0\le xy \le 2\}$. Now come back to the proof that $T$ is concave in $G$. In \eqref{detT} we used that  if $(x, y, A)\in G$, then $xy\le 1$ and the corresponding determinant is non-negative. But the same non-negativity in \eqref{detT} holds under slightly relaxed assumption $(x,y, A)\in \widetilde{G}$.

We notice that our $u_I=\av{u}{I}$, $v_I=\av{v}{I}$, and $A_I= \frac{1}{|I|}\sli_{J\subset I}a_I |I|$ have the dynamics above. The rest of the proof reads exactly as the proof of the Theorem \ref{thenglav}.
\end{proof}

\begin{proof}[Proof of the Theorem \ref{mainineq}]
%We start with a simple Lemma, which follows from the Taylor expansion.
%\begin{lemma}
%Suppose we have three tuples $(N, A)$, $(N_\pm, A_\pm)$, such that:
%$$
%N = \frac{N_+ + N_-}{2}; \; \; \; A = \frac{A_+ + A_-}{2} + m.
%$$
%Moreover, suppose there are $(u, v, L)$, $(u_\pm, v_\pm, L_\pm)$, such that
%$$
%u=\frac{u_+ + u_-}{2}; \; \; \; v=\frac{v_+ + v_-}{2}\;\;\; L=\frac{L_+ + L_-}{2} + m\cdot uv.
%$$
%Then following inequalities hold:
%\begin{align}
%&B_1(N, A) - \frac{B_1(N_+, A_+) + B_1(N_-, A_-)}{2} \geqslant m \cdot \frac{N}{\Psi_0(N)}; \\
%&B_2(u,v,L,A)-\frac{B_2(u_+, v_+, L_+, A_+)+B_2(u_-, v_-, L_-, A_-)}{2}\geqslant 0; \\
%&\mbox{If $L\geqslant \frac{uv}{\ep(\frac1{uv})}$ then} \; B_2(u,v,L,A)-\frac{B_2(u_+, v_+, L_+, A_+)+B_2(u_-, v_-, L_-, A_-)}{2}\geqslant m\cdot uL.
%\end{align}
%\end{lemma}
We start with the following corollary from the Taylor expansion.
Suppose we have three tuples $(N, A)$, $(N_\pm, A_\pm)$, such that:
$$
N = \frac{N_+ + N_-}{2}; \; \; \; A = \frac{A_+ + A_-}{2} + m.
$$
Moreover, suppose there are $(u, v, L)$, $(u_\pm, v_\pm, L_\pm)$, such that
$$
u=\frac{u_+ + u_-}{2}; \; \; \; v=\frac{v_+ + v_-}{2};\;\;\; L=\frac{L_+ + L_-}{2} + m\cdot uv.
$$

Then, since $d^2B_1 \leqslant 0$, we write
$$
B_1(N_+, A_+) \leqslant B_1(N, A)+(B_1)'_N (N, A)(N_+-N) + (B_1)'_A(N, A)(A_+-A).
$$
Thus,
$$
B_1(N,A) - \frac{B_1(N_+, A_+)+B_1(N_-, A_-)}{2}\geqslant (B_1)'_A(N,A) \cdot (A-\frac{A_+ + A_-}{2}) = m\cdot (B_1)'_A(N,A)\geqslant m \frac{N}{\Psi_0(N)}.
$$
Similarly,
$$
B_2(u,v,L,A) - \frac{B_2(u_+, v_+, L_+, A_+)+B_2(u_-, v_-, L_-, A_-)}{2}\geqslant m\cdot ((B_2)'_A(u,v, L, A) + uv(B_2)'_L)
$$

%We continue the proof of the theorem. 

First, suppose that $L_I\leqslant \frac{u_I v_I}{\ep(\frac{1}{u_I v_I})}$. 
Then, using $m=a_I$ we get
\begin{multline}
\mathcal{B}(I)-\frac{\mathcal{B}(I_+)+\mathcal{B}(I_-)}{2} \geqslant \\ \geqslant \ili \left(B_1(N_I(t), A_I) - \frac{B_1(N_{I_+}(t), A_{I_+})+B_1(N_{I_-}(t), A_{I_-})}{2}\right)dt + \\+
\left( B_2(u_I,v_I,L_I,A_I) - \frac{B_2(u_{I_+}, v_{I_+}, L_{I_+}, A_{I_+})+B_2(u_{I_-}, v_{I_-}, L_{I_-}, A_{I_-})}{2}\right)
\\ \geqslant 
 a_I\left((B_2)'_A(u_I,v_I,L_I,A_I) + u_I v_I (B_2)'_L(u_I,v_I,L_I,A_I)\right)+a_I \left( \ili  (B_1)'_A (N_I(t), A_I) dt \right) \geqslant \\
a_I \left(\ili \frac{N_I(t)}{\Psi_0(N_I(t))} dt - \delta_1 u_I L_I\right).
\end{multline}
The last inequality is true, since $(B_2)'_A \geqslant 0$ and $uv(B_2)'_L\geqslant -\delta_1 uL$ on the domain of $B_2$.
We use H\"older's inequality (and that $\int N_I(t)dt = u_I$) to get:
\begin{multline}
\ili \frac{N_I(t)}{\Psi_0(N_I(t))} dt\geqslant \frac{u_I^2}{\ili N_I(t)\Psi_0(N_I(t))dt} \geqslant C \frac{u_I^2}{\ili N_I(t)\Psi(N_I(t))dt\;\ep\!\left(\frac{\ili N_I(t)\Psi(N_I(t))dt}{u_I}\right)}.
\end{multline}
Last inequality is  Theorem \ref{selfimp}.
Therefore, we get that
\begin{multline}
\ili \frac{N_I(t)}{\Psi_0(N_I(t))} dt \geqslant
u_I \cdot \frac{u_I}{\|u\|_{L^\Phi_I}}\cdot \frac{1}{\ep\left(\frac{\|u\|_{L^\Phi_I}}{u_I}\right)}=  u_I \frac{u_Iv_I}{\|u\|_{L^\Phi_I}v_I}\cdot \frac{1}{\ep\left(\frac{\|u\|_{L^\Phi_I}v_I}{u_Iv_I}\right)}.
\end{multline}
We are going to use the one-sided bump condition $\|u\|_{L^\Phi_I}v_I\leqslant B_{u,v}\leqslant 1$. Thus,
$$
u_Iv_I \leqslant \frac{u_I v_I}{\|u\|_{L^\Phi_I}v_I}.
$$
Since the function $x\mapsto \frac{x}{\ep(\frac1x)}$ is increasing near $0$ (on $[0, c_\ep]$) and bounded from below between $c_\ep$ and $1$, we get
$$
\frac{u_Iv_I}{v_I\,\|u\|_{L^\Phi_I}}\cdot \frac{1}{\ep\left(\frac{v_I\,\|u\|_{L^\Phi_I}}{u_Iv_I}\right)}\geqslant C\cdot u_I v_I \frac{1}{\ep(\frac{1}{u_Iv_I})}.
$$
Therefore,
$$
\ili \frac{N_I(t)}{\Psi_0(N_I(t))} dt \geqslant C u_I \frac{u_I v_I}{\ep(\frac{1}{u_I v_I})}\geqslant C u_I L_I.
$$
The last inequality follows from our assumption that $L_I\leqslant \frac{u_I v_I}{\ep(\frac{1}{u_I v_I})}$.
Putting everything together, we get
$$
\mathcal{B}(I)-\frac{\mathcal{B}(I_+)+\mathcal{B}(I_-)}{2} \geqslant a_I u_I L_I (C-\delta_1)\geqslant C_1 \cdot a_I u_I L_I.
$$
\bigskip

We proceed to the case $L_I\geqslant \frac{u_I v_I}{\ep(\frac{1}{u_I v_I})}$. Then we write
$$
\mathcal{B}(I)-\frac{\mathcal{B}(I_+)+\mathcal{B}(I_-)}{2} \geqslant B_2(u_I, v_I, L_I, A_I) - \frac{B_2(u_{I_+}, v_{I_+}, L_{I_+}, A_{I_+})+B_2(u_{I_-}, v_{I_-}, L_{I_-}, A_{I_-})}{2}.
$$
This is obviously true, since $(B_1)'_A\geqslant 0$ everywhere and $B_1$ is a concave function. Next, we use
\begin{multline}
B_2(u_I, v_I, L_I, A_I) - \frac{B_2(u_{I_+}, v_{I_+}, L_{I_+}, A_{I_+})+B_2(u_{I_-}, v_{I_-}, L_{I_-}, A_{I_-})}{2}\geqslant \\
a_I\left((B_2)'_A + uv (B_2)'_L\right)\geqslant ca_I\cdot u_I L_I,
\end{multline}
by the property of $B_2$. Therefore, we are done.
\end{proof}

\section{Fourth step: building the function $B_2$}
\label{B2}
In order to finish the proof, we need to build functions $B_1$ and $B_2$. In this section we will present the function $B_2$. Denote
$$
\vf(x)=\frac{x}{\ep(\frac1x)}.
$$
This function is increasing (by regularity assumptions on $\ep$ in Theorem \ref{main}), therefore, there exists $\vf^{-1}$. We introduce
$$
B_2(u,v,L,A)= Cu - \frac{L^2}{v} \ili_{\frac{A+1}{L}}^{\infty} \vf^{-1}\Big(\frac{1}{x}\Big)\,dx.
$$
Let us explain why the integral is convergent. In fact, using change of variables, we get
$$
\ili_1^{\infty}\vf^{-1}\Big(\frac{1}{x}\Big)\,dx = \ili_0^{\vf^{-1}(1)} \frac{\ep(\frac 1t) - t\frac{d}{dt}(\ep(\frac 1t))}{t}\,dt,
$$
which converges at $0$ by assumption \eqref{intep}.

Therefore, since $L_I\leqslant C \sqrt{u_I v_I}$, we get
$$
0\leqslant B(u_I, v_I, L_I, A_I) \leqslant Cu_I. 
$$

\bigskip

Next,
%$$
%(B_2)'_A = \frac{L}{v}\vf^{-1}(\frac{L}{A+1})\geqslant 0.
%$$
%Moreover,  we get
\begin{multline}
(B_2)'_A + uv(B_2)'_L = \frac{L}{v}\vf^{-1}\Big(\frac L{A+1}\Big) - u(A+1)\vf^{-1}\Big(\frac L{A+1}\Big) - 2uL\ili_{\frac{A+1}{L}}^{\infty}\vf^{-1}\Big(\frac 1x\Big)\,dx =
 \\
 uL \cdot\left(\frac{1}{uv}\vf^{-1}\Big(\frac L{A+1}\Big) - \frac{A+1}{L}\vf^{-1}\Big(\frac L{A+1}\Big) - 2\ili_{\frac{A+1}{L}}^{\infty}\vf^{-1}\Big(\frac 1x\Big)\,dx \right)
\end{multline}
We use that $L\geqslant \frac{uv}{\ep\big(\frac{1}{uv}\big)}=\vf(uv)$. Then $\vf^{-1}(L)\geqslant uv$, and, since $A+1\sim 1$, we get
$$
\frac{1}{uv}\vf^{-1}\Big(\frac L{A+1}\Big) \geqslant C_1.
$$
Moreover, since $uv\leqslant \delta$ is a small number, we get that $L$ is small enough for the integral $\ili_{\frac{A+1}{L}}^{\infty}\vf^{-1}(\frac 1x)dx$ to be less than a small number $c_2$. Finally, let us compare $\frac{A+1}{L}\vf^{-1}(\frac L{A+1})$ with a small number $c_3$. Since $L$ is small, we can write
$$
\ep\Big(\frac 1{c_3 L}\Big) \leqslant c_3.
$$
We do it, since $c_3$ is fixed from the beginning (say, $c_3=\frac{1}{10}$). Thus,
$$
L\leqslant \vf(c_3L).
$$
This implies
$$
\vf^{-1}(L)\leqslant c_3 L,
$$
thus
$$
\frac1L \vf^{-1}(L)\leqslant c_3.
$$
Since $A+1\sim 1$, we get the desired.
Therefore, if $L\geqslant \frac{uv}{\ep(\frac{1}{uv})}=\vf(uv)$ then $(B_2)'_A + uv(B_2)'_L \geqslant cuL$. 

Moreover, in the whole domain of $B_2$ we get, since $(B_2)'_A \geqslant 0$, 
$$
(B_2)'_A + uv(B_2)'_L \geqslant uv (B_2)'_L \geqslant -(c_2+c_3)uL
$$
with small $c_2+c_3$. This is a penultimate inequality in the statement of Theorem \ref{mainineq}.

Now we shall prove the concavity of $B_2$.  For this it is enough to prove the concavity of the function of three variables: $B(v, L, A):= B_2(u,v, L, A) -Cu$. Clearly, $(B)''_{vv} < 0$, which is obvious. 
Also, it is a calculation that
$$
\det \begin{pmatrix}
&(B)''_{vv} &&(B)''_{vA} &&&(B)''_{vL}\\
&(B)''_{vA} &&(B)''_{AA} &&&(B)''_{AL}\\
&(B)''_{vL} &&(B)''_{AL} &&&(B)''_{LL}
\end{pmatrix}
=0.
$$
Thus, we need to consider the matrix
$$
\begin{pmatrix}
&(B)''_{vv} &&(B)''_{vA} \\
&(B)''_{vA} &&(B)''_{AA}
\end{pmatrix}
$$
and to prove that its determinant is positive.
We denote $f(t)=\vf^{-1}(t)$, to simplify the next formula.
The calculation shows that the determinant above is equal to
$$
g\Big(\frac{L}{A+1}\Big):=-f\Big(\frac{L}{A+1}\Big)^2 + 2\Big(\frac{L}{A+1}\Big)^2\cdot f'\Big(\frac{L}{A+1}\Big) \ili_{\frac{A+1}{L}}^{\infty} f\Big(\frac1x\Big)\,dx.
$$
We need to prove that $g$ is positive near $0$. First, $g(0)=0$. Next,
\begin{multline}
g'(s) = -2 f(s)f'(s) + 4s f'(s) \ili_{\frac1s}^{\infty} f\Big(\frac1x\Big)\,dx + 2s^2 f''(s)\ili_{\frac1s}^{\infty} f\Big(\frac1x\Big)\,dx + 2 f'(s) f(s) = 
\\
4s f'(s) \ili_{\frac1s}^{\infty} f\Big(\frac1x\Big)\,dx + 2s^2 f''(s)\ili_{\frac1s}^{\infty} f\Big(\frac1x\Big)\,dx.
\end{multline}

We notice that $f'$ is positive, since $\vf^{-1}$ is increasing near $0$. Moreover, by the fact that $\vf$ is  strictly monotonous,  and by concavity of $t\ep(t)$ (see Theorem \ref{propep}), we get that  $\vf$ is {\it strictly} convex, hence $\vf^{-1}$ is  strictly convex near $0$ as well. That is, $f''$ is also positive. 

Therefore, $g'(s)>0$, and so $g(s)> g(0)=0$.   The application of Lemma \ref{Sy} finishes the proof of concavity of $B$ (and therefore  of the concavity of $B_2$). We are done.

\begin{zamech}
We can always think that the bump constant $B_{u,v}\leqslant C_\ep$, where $C_\ep$ is such that $L_I\leqslant c_\ep$. Then we can use the monotonicity and concavity of the function $\vf$ near $0$.
%We do everything near $0$, since instead of $A+1$ we could put $A+C_\ep$, where $C_\ep$ is such that $\frac{L}{A+1}\leqslant c_\ep$, and all functions of this argument are increasing and concave.
%
%Or, we can always think that $u_I v_I \leqslant C_\ep$, where this constant is such that $L_I\leqslant c_\ep$, which is again enough.
\end{zamech}

\section{Fifth step: building the function $B_1$}
\label{B1}
We present the function from \cite{NRTV}.
$$
B_1(N, A) = CN - N\ili_{0}^{\frac NA}\frac{ds}{s\Psi_0(s)}
$$

\section{Why function $\ep$ was so needed?}
\label{whyB2}

Integrability condition \eqref{intep} on function $\ep$ was used in constructing $B_2$ in a very essential way. A natural question arises, why not to get rid of $\ep$? 
Suppose  we can build function $B$ in the domain $\widehat\Omega$ such that
$$
\widehat\Om = \{(u,v,L,A)\colon 0\leqslant A \leqslant 1; \; u,v,L \geqslant 0; \; uv\leqslant \delta; \; L\leqslant P\cdot \sqrt{uv}\}.
$$
\begin{equation}
\label{hypoB}
\begin{aligned}
%& 0\leqslant B_1\leqslant N;\\
%& (B_1)'_A \geqslant 10\frac{N}{\Psi_0(N)};\\
%& (B_1)'_A \geqslant 0; \\
%& -d^2 \,B_1 \geqslant 0;\\
& 0\leqslant B\leqslant u;\\
& (B)'_A \geqslant 0\\ 
&(B)'_A\geqslant c\cdot u\cdot L, \; \mbox{when} \; P \sqrt{uv}\ge L\geqslant uv; \\
&uv(B)'_L \geqslant -\delta_1 uL, \; \mbox{for sufficiently small $\delta_1$ in the whole of $\widehat\Omega$}; \\
& -d^2\,B \geqslant 0.
\end{aligned}
\end{equation}
Looking at the proof of Theorem \ref{mainineq} we immediately see that this $B$ can replace our $B_2$ in this proof and, thus, give us \eqref{glav}
{\it without} any extra conditions on $\Phi$ or corresponding $\Psi$ apart a necessary condition of integrability: $\int^\infty\frac{dt}{\Phi(t)}<\infty$.

Hence, the existence of such a function would prove (as we have explained before) the one-sided bump conjecture in full generality. 

\medskip

\begin{zamech}
However, we are inconclusive whether $B$ as in \eqref{hypoB} exists. We ``almost" prove below that it does not exist.
\end{zamech}

\medskip

Put
$$
\Om := \{(u,v, L, A)\colon 0\le A\le 1;\; u, v, L\ge 0;\; uv \le 1; P\sqrt{uv} \ge L\ge uv\}\,,
$$
and
$$
\,\Om_0: = \{(u,v, A)\colon 0\le A\le 1;\; u, v \ge 0;\; uv \le 1\}\,.
$$

\begin{theorem}
\label{shep}
It is impossible to find a smooth function $B$, defined on $\Om$, such that:
\begin{align}
& 0\leqslant B \leqslant u;\\
%& B'_A \geqslant 0\\ 
&B'_A \geqslant c\cdot u\cdot L, \; \mbox{when} \; P\cdot \sqrt{uv}\ge L\geqslant uv; 
%&uv(B_2)'_L \geqslant -\delta_1 uL, \; \mbox{for sufficiently small $\delta_1$}; 
\\
& -d^2\,B \geqslant 0;
\end{align}
which satisfies one extra condition:
\begin{equation}
\label{concB0}
B_0(u,v, A) := \sup_{L\colon uv \le L\le P\sqrt{uv}} B(u,v, L, A) \;\;\;\text{is concave in}\;\;\;\Om_0 \,.
\end{equation}
\end{theorem}

The proof will consist of two parts. First we show that if function $B$ in  Theorem \ref{shep} exists, then a certain other function must exist.
Only then we come to a contradiction with the existence of this new function built in the lemma that now follows.

\begin{lemma}
\label{B0}
Given a smooth function $B$ from Theorem \ref{shep} one can build a function $B_0$ such that it is defined in $\Om_0$
and such that
\begin{align}
& 0\leqslant B_0 \leqslant u;\\
%& B'_A \geqslant 0\\ 
&(B_0)'_A \geqslant c\cdot u^2\cdot v ; 
%&uv(B_2)'_L \geqslant -\delta_1 uL, \; \mbox{for sufficiently small $\delta_1$}; 
\\
& -d^2\,B_0 \geqslant 0.
\end{align}
\end{lemma}

\begin{proof}
Given $B$ in $\Omega$ consider a new function defined on $\Omega_0$:
$$
B_0(u, v, A) :=\max_{L\colon\;(u,v, A, L)\in \Omega}B(u,v, L, A) = \max_{L\colon \; uv \le L\le P\,\sqrt{uv}}B(u,v, L, A) \,.
$$
Call the point, where the maximum is attained $L(u,v, A)$. Fix $A, u$. Let set $I_{A, u}$ be the set of $v\in [0, 1/u]$ 
such that this maximum is attained strictly inside: $L(u, v, A)\in (uv, P\sqrt{uv})$.
Then for $v\in I_{A,u}$ we have
$ \frac{\pd B}{\pd L} (u, v, L(u,v, A), A) =0$.  Consequently
 \begin{equation}
 \label{rav}
 \frac{\pd B_0}{\pd A} =   \frac{\pd B}{\pd A} (u, v, L(u,v, A), A) +  \frac{\pd B}{\pd L}(u, v, L(u,v, A), A)\cdot \frac{\pd L}{\pd A}(u,v, A) \,.
 \end{equation}
 Using the middle property of $B$ from Theorem \ref{shep} we get
  \begin{equation}
  \label{oce}
 \frac{\pd B_0}{\pd A}  \ge c \cdot u\cdot uv = c\cdot u^2v
 \end{equation}
 to be satisfied on the closure of $I_{A,u}$ for all $A, u$. If this closure is not the whole $[0,1/u]$, 
 then its complement contains a sub-interval of $[0,1/u]$. Call this sub-interval $S$. On $S$ maximum in the definition of 
 $B_0$ is attained either for $L(u,v, A)= uv$ or $L(u,v, A)=P\,\sqrt{uv}$.  It is easy to see that we can 
 fix one of this choices (may be by making $S$ smaller).
 So suppose maximum is attained for $v\in S_{A,u}$ for $L(u,v, A)= uv$. Let us start to vary $A,u$ a bit. 
 We will see that the set of $A,u$, where the closure of $I_{A,u}$ is not the whole $[0,1/u]$ is an open set $G$.
 Foliating $\{(u, v,  \,A)\colon (A,u)\in G, v\in S_{A,u}\}$ by surfaces $uv=const$ we see that $\frac{\pd L(u,v, A)}{\pd A} =0$ on $G$.
 
 Then again $\frac{\pd B}{\pd L}(u, v, L(u,v, A), A)\cdot \frac{\pd L}{\pd A}(u,v, A) =0$, and \eqref{oce} is thus  valid because of  equality \eqref{rav}.

We are left to prove the concavity of $B_0$: $-d^2 B_0 \ge 0$.  But this concavity in guaranteed by our extra requirement \eqref{concB0}.
\end{proof}

\begin{zamech}
Function $B_0$ is a supremum of concave functions, and as such is not {\it automatically} concave. Only the infimum of concave functions is automatically concave. Therefore our extra requirement \eqref{concB0} seems very ad hoc and too strong. In fact, it is ``almost" no requirement at all. We cannot get rid of the word ``almost" though. However, let us explain that there is a very generic common situation when the supremum of the family of concave functions is indeed concave. 

It is really easy to see that when one has a concave function of several variables given in the {\it convex} domain, and one forms the supremum of it over varying one of its variables with all other variables fixed, then one gets a concave function again. 

We are doing ``almost" that in our construction of $B_0$ from $B$. The difference is that $\Omega$ is not convex anymore! And this is the only obstacle for deleteing the word ``almost" above. If not for that small obstacle the next lemma would prove that the set of functions having properties \eqref{hypoB} is empty.
\end{zamech}

Now we prove 
\begin{lemma}
\label{B0nonexi}
Function $B_0$ such that $-d^2\,B_0 \geqslant 0$,  $(B_0)'_A \geqslant c\cdot u^2\cdot v, c >0$,   $ 0\leqslant B_0 \leqslant u$ in $\Omega_0=\{(u,v, A)\colon 0\le A\le 1;\; u, v \ge 0;\; uv \le 1\}$ does not exist.
\end{lemma}

\begin{proof}
Suppose it does exist. Take any sequence $\{\al_I\}, 0\le \al_I\le 1,$ enumerated by
dyadic lattice of $I_0:=[0,1]$ such that for any $J\in \cD$
\begin{equation}
\label{carse}
A_J:=\frac1{|J|}\sum_{I\in\cD\colon I\subset J} \al_I |I| \le 1\,.
\end{equation}
(Carleson property.)
Take any two functions $u, v$ on $[0,1]$  such that
\begin{equation}
\label{avav}
\av{u}{I}\cdot \av{v}{I} \le 1,\;\forall I\in\cD\,.
\end{equation}
Consider 
$$
\cB(I):= B_0 (\av{u}{I}, \av{v}{I}, A_I)\,.
$$
It is then easy to see using the properties of $B_)$ and Taylor's formula that
\begin{equation}
\label{ind10}
|I| \cB(I) - (|I_+| \cB(I_+) + |I_-| \cB(I_-) )\ge c\cdot\al_I |I|\cdot \av{u}{I}^2\cdot \av{v}{I} \,.
\end{equation}

Summing this up, we get
\begin{equation}
\label{ind11}
c\sum_{I\in \cD, \, I\subset I_0} \av{u}{I}^2\cdot \av{v}{I}\al_I |I| \le\cB(I_0) \le \int_0^1 u \,dx\,.
\end{equation}

\medskip

Now we construct functions $u,v$ and Carleson sequence $\al_I$ such that they satisfy the properties just mentioned, but such that 
\eqref{ind11} fails.  To do that we choose $u\ge 0$ on $I_0$  whose specifications will be made later. Let
$\cG_n$ be the family of maxima dyadic intervals inside $I_0$ such that
$$
\av{u}{I}\ge 3^n\,.
$$
Here are two facts, firstly:
\begin{equation}
\label{sv}
\forall I\in \cG_n\;\; \av{u}{I} \le 2\cdot 3^n\,,
\end{equation}
and secondly, if for any $J\in \cG_{n-1}$ we denote $\cG_n(J)$ those $I\in \cG_n$ that lie inside $J$, then
\begin{equation}
\label{sn}
\frac{|J\setminus\cup_{I\in \cG_n(J)} I |} {|J|} \ge \frac13\,.
\end{equation}
Inequality \eqref{sv} is obvious by maximality of intervals. Inequality \eqref{sn} the followslike that:
$$
\frac{\cup_{I\in \cG_n(J)}| I | \cdot 3^n }{|J|} \le \frac{\sum_{I\in \cG_n(J)}\int_I u\,dx }{|J|} \le \frac{\int_J u\,dx}{|J|} =\av{u}{J} \le 2\cdot 3^{n-1}\,.
$$
Therefore,
$$
\frac{\cup_{I\in \cG_n(J)}| I |  }{|J|}\le \frac23\,,
$$
and \eqref{sn} is proved.

Before choosing $u$ we will now choose $v$.  We build $v$ ``from bottom to top". 
Choose a large $n$, and on each $I\in \cG_n$  choose $v$ to be the same constant $3^{-n}$. At this moment \eqref{avav} is of course satisfied.

Now we consider $J\in \cG_{n-1}$, and we want to keep \eqref{avav} for this $J$.
If we would keep $v$ to be $3^{-n}$ on all $I\in\cG_n(J)$ and we put $v=3^{-n}$ on $J\setminus\cup_{I\in \cG_n(J)} I $, 
then notice that $\av{u}{J}$ drops $3$ times with respect to $\av{u}{I}$, but $\av{v}{J}$ does not drop 
with respect to $\av{v}{I}$, we would see that the product drops $3$ times. However, we want it not to drop at all. So we keep 
$v$ to be $3^{-n}$ on all $I\in\cG_n(J)$ and we put $v= 9\cdot 3^{-n}$ on $J\setminus\cup_{I\in \cG_n(J)} I $. 
This portion is at least $\frac13 |J|$. Therefore, even without extra help from $v$ on all $I\in\cG_n(J)$ we would have at this moment $\av{u}{J}\av{v}{J}\ge 1$. 
We want exactly $1$. So we choose $c\in (1,9)$ such that if $v$ on all $I\in\cG_n(J)$ is $3^{-n}$  and on $J\setminus\cup_{I\in \cG_n(J)} I $ it is $c\cdot 3^{-n}$, then
\begin{equation}
\label{avJ}
\av{u}{J}\av{v}{J} =1\,.
\end{equation}

On the top of that we have an absolute estimate for all $L\in\cD$ such that $L\subset J$ and $L$ not a subset of $\cup_{I\in \cG_n(J)} I $. In fact, 
for such $L$ we have $\av{u}{L} \le 3^n, \av{v}{L} \le 9\cdot 3^{-n}$, hence
\begin{equation}
\label{avL}
\av{u}{L}\av{v}{L} \le 9\,.
\end{equation}
 Now we already built $v$ on every $J\in \cG_{n-1}$. The passage from $\cG_{n-1}$ to $\cG_n$ repeats those steps. Finally we will be finishing with $v$ such that
 \eqref{avJ} holds for all $J\in \bigcup_{k=1}^n\cG_k$, and \eqref{avL} holds for all $L\in\cD$ such that $L$ is not equal to or being  inside of  any of the intervals $I\in \cG_n$.
 
 Making $n\to\infty$ we have $v$  such that
 \eqref{avJ} holds for all $J\in \bigcup_{k=1}^\infty\cG_k$, and \eqref{avL} holds for all $L\in\cD$.
 
 Now we notice two things: we can multiply $v$ by $1/9$ to have $1$ in the right hand side of \eqref{avL}, and we can put
 $$
 \al_I := \begin{cases} 1/3\,,\;\; I \in \bigcup_{k=1}^\infty\cG_k
 \\
 0\,,\;\;\text{otherwise}\,.
 \end{cases}
$$

Then by \eqref{sn} it is a Carleson sequence in the sense of \eqref{carse}. Now let us disprove \eqref{ind11} by the choice of $u$. The left hand side of \eqref{ind11} now can be written (by looking at \eqref{avJ}) as
$$
\sum_{I \in \bigcup_{k=1}^\infty\cG_k}\av{u}{I}^2\av{v}{I} \al_I |I| = \frac1{27} \sum_{k=1}^\infty\sum_{I \in \cG_k}\av{u}{I} |I|\,.
$$

Let for those $I$ consider $E_I:= I\setminus \bigcup_{\ell\in \cG_{k+1}}\ell$.  Let $M^d$ means dyadic maximal operator. Then on those $I$ we have $M^d u\le 3^{k+1}$, $\av{u}{I}\ge 3^k$, and $E_I$ tile $I_0$.
Hence
$$
\sum_{k=1}^\infty\sum_{I \in \cG_k}\av{u}{I} |I|\ge \sum_{k=1}^\infty\sum_{I \in \cG_k}\av{u}{I} |E_I|\ge \frac13 \int_{I_0} (M^d u)(x)\, dx\,.
$$
Here is the choice of $u$ that finally violates \eqref{ind11}: $u$ must be such that $\int_0^2 M^d u\, dx =\infty, \int_0^1 u\, dx <\infty$. Lemma is proved.
\end{proof}

Theorem \ref{shep} is proved.

\section{Appendix: translating from \cite{CURV} to our language}
\label{appendix}
For simplicity we focus only on the case $p=2$. In \cite{CURV} the authors considered a bump funtcion $A(t)$ and studied the norm $\|u^{\frac12}\|_A$. In our language, $\Phi(t)=A(t^{\frac12})$. Then we notice that
$$
\|u^{\frac12}\|_A = \|u\|_{\Phi}^\frac12.
$$
This follows from the definition of the Orlitz norm. 

Suppose now that we chose a function $\Phi_0$ with $\Psi_0(s)\leqslant \Psi(s)\ep(\Psi(s))$. We have peoved that it implies the following improvement of Orlitz norm:
$$
\|u\|_{\Phi_0}\leqslant C \|u\|_{\Phi} \cdot \ep\bigg(\frac{\|u\|_{\Phi}}{\ave{u}}\bigg).
$$
Translating it to the language of \cite{CURV}, we get (after taking square roots of both sides)
$$
\|u^\frac12\|_{A_0} \leqslant C \|u^{\frac12}\|_A \cdot \left(\ep\bigg(\frac{\|u^{\frac12}\|_{A}^2}{\|u^{\frac12}\|_{L^2}^2}\bigg)\right)^{\frac12}.
$$
Thus, the Orlisz norms for $A$ and $A_0$ improve in the following way:
$$
\|u^\frac12\|_{A_0} \leqslant C \|u^{\frac12}\|_A \cdot \ep_{A_0, A}\bigg(\frac{\|u^{\frac12}\|_{A}}{\|u^{\frac12}\|_{L^2}}\bigg), 
$$
where
$$
\ep_{A_0, A}(t) = \sqrt{\ep(t^2)}.
$$

Thus, our integral condition on $\ep$ gives the following condition on $\ep_{A_0, A}$:
$$
\ili^{\infty} \frac{\ep_{A_0, A}(\sqrt{t})^2}{t}dt\leqslant \infty.
$$
We denote $y=\sqrt{t}$, thus $dy/y = dt/t$. And we get
$$
\ili^{\infty} \frac{\ep_{A_0, A}(y)^2}{y}dt <\infty.
$$

However, the proof of the results from \cite{CURV} gives a condition
$$
\ili^{\infty} \frac{\ep_{A_0, A}(y)}{y}dt <\infty.
$$
We notice that $\ep_{A_0, A}(y)$ is small at infinity, and thus $\ep_{A_0, A}(y)^2 < \ep_{A_0, A}$. Therefore, our result gives the result of \cite{CURV} and improves it.

\end{document}